%% file: RadiusStatus.tex
\documentclass[a4paper]{article}

\usepackage[english]{babel}
\usepackage[utf8]{inputenc}
\usepackage[T1]{fontenc}
\usepackage{amsmath,amsfonts,amsthm,amssymb}
\usepackage[small,bf]{caption}
\usepackage{pgf}
\usepackage{subfig}

\usepackage[
pdftitle={Bounds on the radius and status of graphs},
pdfsubject={graph theory},
pdfauthor={Roswitha Rissner, Rainer E. Burkard},
pdfkeywords={graph theory, status, radius, midpoints, centre, centroid, median},
bookmarksopen=true,
bookmarksnumbered=true,
pdfpagemode=UseOutlines,
pdfdisplaydoctitle=true
]{hyperref}

\hypersetup{colorlinks=true, linkcolor=darkgray, citecolor=darkgray, urlcolor=darkgray}

\newcommand{\dist}{d}
\newcommand{\distT}{d_T}
\DeclareMathOperator{\rad}{rad}
\DeclareMathOperator{\ecc}{ecc}
\newcommand{\s}{s}
\DeclareMathOperator{\W}{\mathbf{w}}

\newcommand{\T}{\overline{T}}
\newcommand{\pic}{pgfpicture}

\theoremstyle{plain}
\newtheorem{The}{Theorem}[section]
\newtheorem{Lem}[The]{Lemma}
\newtheorem{Pro}[The]{Proposition}

\theoremstyle{definition}

\newtheorem*{Exa}{Example}

\newtheorem*{Beh}{Claim}

\numberwithin{equation}{section}
\begin{document}
\title{ Bounds on the radius and status of graphs}
\author{
Roswitha Rissner\footnotemark[2]
\and
Rainer E.\ Burkard\footnotemark[1]
}

\def\thefootnote{\fnsymbol{footnote}}
\footnotetext[1]{Graz University of Technology, Institute of Optimization and Discrete Mathematics, Steyrergasse 30, Graz,
Austria. \texttt{burkard@tugraz.at} }
\footnotetext[2]{Graz University of Technology, Institute of Analysis and Computational Number Theory, Steyrergasse 30, Graz,
Austria. \texttt{roswitha.rissner@tugraz.at} }

\footnotetext{This research is supported by the Austrian Science Fund (FWF): W1230, Doctoral Program ``Discrete Mathematics'' }

\date{November, 2013}

\maketitle

\input{Abstract}
\input{Introduction}
\input{Transformation}
\input{lowerBound}
\input{upperBound}
\input{boundsInGeneralGraphs}

\input{conclusions}

\input{acknowledgement}

\bibliographystyle{abbrv}
\bibliography{bibliography}

\end{document}

%% file: Abstract.tex
\begin{abstract}

Two classical concepts of centrality in a graph are the median and the center. 
The connected notions of the  status and the radius of a graph 
seem to be in no  relation.
In this paper, however,  we show a clear connection of both concepts, as they obtain their 
minimum and maximum values at the same type of tree graphs. 
Trees with fixed maximum degree and extremum radius and
status, resp., are characterized. The bounds on radius and status can be transferred to general connected graphs via spanning trees.

A new method of proof allows not 
only to regain results of Lin et al.~on graphs with
extremum status, but it allows also to prove analogous results on
graphs with extremum radius. 

\end{abstract}

%% file: Introduction.tex
\section{Introduction}

The status and radius of a graph are fundamental notions in graph theory. 
This paper presents upper and lower bounds for both with respect to the order and the maximum degree of the graph. 
We will characterize trees with fixed order and fixed
maximum degree and  minimum (maximum) radius and status. 
It turns out that such trees have the same structure for the radius and the
status although the positions of the center and the centroid do not coincide in
general. Therefore the results demonstrate that despite several different 
 properties of center and centroid these two notions correlate in a subtle way.

Before we give an overview on the structure of this work, we want to 
clarify the notation and recall the definitions used below.
Throughout this paper a graph $G = (V,E)$ is undirected, 
connected and simple (no loops, no multiple edges). The \textit{order}  
of $G$ is the cardinality $|V|$ of its vertex set $V$ and 
the \textit{maximum degree} $\Delta(G)$ of $G$ is defined as 
$\max_{x\in V} \deg(x)$, the largest degree of a vertex of $G$.
The \textit{distance} $\dist(u,v)$ between
two vertices $u$ and $v$ is defined as 
the number of edges of a 
shortest path from $u$ to $v$.

 The \textit{status} $s(x)$ of a vertex $x$ is
defined as the sum of distances to all other vertices of the
graph, that is,
\begin{align}\label{statusofx}
 \s(x) = \sum_{y \in V} \dist(x,y).
\end{align}
The status $s(G)$ of a graph $G$ is the minimum status of a vertex
of $G$, that is,
\begin{align*}
 \s(G) = \min_{x \in V} \s(x).
\end{align*}
A vertex of minimum status is called \textit{median}. This notion is strongly 
related to the notion of a \textit{centroid vertex}, that is, a vertex (of a tree) of minimum 
branch weight. We will recall the exact definition in Section \ref{sec:transformation}.
In fact, Zelinka \cite{Zelinka68} showed that a vertex of a tree is a median if and
only if it is a centroid vertex (in his work a centroid vertex  is called mass center). Therefore we will (in case the underlying graph is a tree) refer to a median also
as a \textit{centroid} vertex. The set of centroid vertices of a tree is called \textit{centroid} 
of the tree.

Replacing the sum in \eqref{statusofx} by the maximum 
leads us to the \textit{eccentricity} of a vertex and the 
\textit{radius} of a graph.
Thus the eccentricity $\ecc(x)$ of a vertex $x$ is the maximum
distance to a vertex in $G$, that is,
\begin{align*}
 \ecc(x) = \max_{y \in V} \dist(x,y).
\end{align*}
The radius $\rad(G)$ of  graph $G$ is the minimum eccentricity of a
vertex of $G$, that is,
\begin{align*}
 \rad(G) = \min_{x \in V} \ecc(x).
\end{align*}

A vertex of minimum eccentricity is called a \textit{central vertex}, the set of central vertices is called the \textit{center} of the graph.
Further research on these terms and generalizations can be found  e.g. in the work of Hakimi (\cite{Hakimi64},\cite{Hakimi65}), Kariv and Hakimi (\cite{KarivHakimi79},\cite{KarivHakimi79_1}), Jeger and Kariv \cite{JegerKariv85}, Tansel, Francis and Lowe \cite{TanFranLow83}, 
Lin and Shang (\cite{LinShang09}) and Lin et al.~(\cite{LinShangZhang11}).

Before we state results for general graphs in the Section \ref{sec:generalgraphs} 
we will concentrate on trees. Sections \ref{sec:lower} and \ref{sec:upper} 
are dedicated to details on lower and upper bounds for the  
radius and status of trees. 
These results will, together with a lemma, imply results 
on general graphs in the Section \ref{sec:generalgraphs}. 
This is possible due to the observation that a connected graph always contains 
a spanning tree which allows us to reduce the argumentation to trees. 
In Section \ref{sec:concl} we conclude our work and give a small outlook on a 
possible generalization of this theory in case of a weighted graph.

We want to stress, that the results on the status of a graph were already published 
by Lin et al.~\cite{LinShangZhang11}. The authors used 
chains of inequalities to show their statements. 
Further the value of the upper bound of the radius was already shown by Vizing (\cite[Lemma 1]{Viz67}).
We introduce in this paper a new proof technique, 
based on tree-transformations which not only points out the similar
behavior of radius and status in dependence of order and maximum
degree of the input tree, but also allows to regain the results of
Lin et al.~in an elegant way.

This transformation will be introduced in Section \ref{sec:transformation} and,
in particular, the location of the centroid of the transformed
tree is discussed.

This paper emerged from a thesis by Rissner
\cite{Rissner11} at Graz University of Technology.

%% file: Transformation.tex
\section{Transformation}
\label{sec:transformation}

In this section we exhibit a transformation for trees: if a tree
with given order and maximum degree does not have an optimum (minimum or
maximum, resp.) status or radius then we can apply a simple
transformation to obtain a tree of the same order and the same
maximum degree, but with a better (lower or larger) value for status or radius,
respectively.

The transformation consists of the reallocation of a single leaf, 
that is, a vertex of degree 1. Clearly this transformation can change the value 
of the radius at most by 1 and it is rather easy to determine if it is increasing or decreasing 
depending on the choice of the leaf. 
However, the change rate of the status is not that easily determined. To compute the status it is 
useful to know a centroid vertex of the tree. In this section we investigate the position of the centroid 
of the transformed tree with respect to a centroid vertex of the original tree. Given a centroid vertex 
$x$ of the original tree, it turns out that if $x$ is not a centroid vertex of the transformed tree, then a 
neighbor of $x$ is, and we can determine which neighbor in dependence of the location of the reallocated leaf.

In this context the classical definition of the centroid turned out to be useful, 
hence we will recall it at this point. Let $T$ be a tree and 
$x$ a vertex, a \textit{branch}
$T'$ at $x$ is a maximal subtree of $T$ which contains $x$ as a
leaf. The \textit{weight} of a branch  is defined as $|T'| - 1$,
that is, the number of vertices in this branch excluding the vertex
$x$. The \textit{branch weight} $\W_T(x)$ of $x$ is the maximum weight of a 
branch at $x$ and a centroid vertex is a vertex of minimum branch weight. 

This section provides two propositions which describe the location of 
a centroid vertex of the transformed tree with respect to the position 
of a centroid vertex of the original tree.
We start with the following well-known lemma.

\begin{Lem}\label{centroid}
 Let $T$ be a tree of order $n$ and $v$ a vertex of $T$. Then $v$ is a centroid vertex of $T$ if and only if its branch weight $\W_T(v)$ is less than or equal to $\frac{n}{2}$.
\end{Lem}

An immediate consequence of this lemma is the following inequality.
Let $T'$ be a branch at a centroid vertex $x$, then 
\begin{align}\label{eq:centroidInequs}
  |T'| \leq |T \setminus T'|  +1.
\end{align}

\begin{Pro}\label{transformationCentroid1}
Let $T = (V,E)$ be a tree and let $x$ be a centroid vertex, $b$ a
leaf, $\bar b$ the vertex adjacent to $b$ and $u \neq x$ an
arbitrary vertex of $T$. Further let $T_1$ be the branch at $x$
containing $u$, $T_2$ the branch containing $b$ and  $S = (T
\setminus (T_1 \cup T_2)) \cup \{x\}$ the union of all other
branches at $x$.

For $\overline{T} = (V , (E \setminus \{(\bar b,b)\}) \cup
\{(u,b)\})$, the tree resulting from $T$ by removing the edge
$(\bar b,b)$ and inserting the edge $(u,b)$ instead, the following
statements hold.

\begin{enumerate}
 \item[(a)] If $T_1 = T_2$ or $|T_1| < |T_2| + |S| - 1$, 
            then $x$ is a centroid vertex of $\overline{T}$.
 \item[(b)] If $T_1 \neq T_2$ and $|T_1| \geq |T_2| + |S| - 1$, 
            the vertex adjacent to $x$ on the shortest path (geodesic) 
            from $x$ to $u$ is a centroid vertex of $\overline{T}$.
\end{enumerate}

\end{Pro}

\begin{proof}
In case $T_1 = T_2$, the branches  of $\overline{T}$ at $x$ have the same 
weights as the branches of $T$. Due to Lemma \ref{centroid}, $x$ is a centroid vertex of $\overline{T}$.

Henceforth let us consider the case $T_1 \neq T_2$. Depending on
the degree of $x$, the subtree $S$ contains no, exactly one or
more than one branch at $x$. Let $S'\subset S$ be a branch at $x$
of maximum weight in $S$. In case there is no branch inside $S$, let
$S'= S = \{x\}$ be the subtree consisting of the single vertex $x$.
Let further $\overline{T}_1 = T_1 \cup \{b\}$ and $\overline{T}_2
= T_2 \setminus \{b\}$ be the branches of $\overline{T}$
corresponding to the branches $T_1$ and $T_2$, respectively.

The branch weight of $x$ in $\overline{T}$ is equal to
\begin{align*}
 \W_{\overline{T}}(x) = \max\{|\overline{T}_1|,|\overline{T}_2|,|S'|\} - 1.
\end{align*}

If $\W_{\overline{T}}(x)\leq \W_T(x)$ then $x$ is a centroid vertex of $\overline{T}$ due to 
Lemma \ref{centroid}. So, we need to check what happens in case the branch weight of $x$ increases. This is the case 
if $T_1$ is a branch of $x$ of maximum weight (in $T$ and thus in $\overline{T}$), that is, 
\begin{align*}
 |T_1| \geq \max\{|T_2|, |S'|\}
\end{align*}
and we have 
\begin{align}\label{eq:weight}
 \W_{\overline{T}}(x) = |\overline{T}_1|- 1 = |T_1|.
\end{align}

After these observations it is quite easy to prove the first claim.
\begin{Beh}
The subtree $\overline{T}_1$ contains at least one centroid vertex of $\overline{T}$.
\end{Beh}

\input{transformation_2}

Let $z$ be a neighbor of $x$ in $\overline{T}_2$ (see Fig.~\ref{fig:notInT2}).
The subtree $\overline{T}_1 \cup S \cup \{z\}$
with weight $|\overline{T}_1| + |S| -1 = |T_1| + |S|$ is a branch at $z$ of maximum weight, hence we get
\begin{align*}
 \W_{\overline{T}} = |T_1| + |S| > |T_1| = \W_T(x).
\end{align*}

Now, assume that $z$ is a neighbor of $x$ in $S$ and $\bar S$ is the branch at $x$ containing $z$ (see Fig.~\ref{fig:notInS}). A branch of maximum weight at $z$ in $\overline{T}$ is 
$\overline{T}_1 \cup \overline{T}_2 \cup (S \setminus \bar S)\cup \{z\}$, its weight amounts to $|\overline{T}_1| + |\overline{T}_2|+ |S| - |\bar S| -1$. 
Again, we get 
\begin{align*}
 \W_{\overline{T}} = |\overline{T}_1| + |\overline{T}_2|+ |S| - |\bar S| -1 \geq |T_1| = \W_T(x).
\end{align*}

\input{transformation_3}

\noindent 
The branch weights of all vertices in $S \cup \overline{T_2}$ must therefore be greater than or equal to the branch weight of $x$, that is, either $x$ is a centroid vertex of $\overline T$ or some other vertex of $\overline{T_1}$. In both cases a centroid vertex is located in $\overline{T_1}$.

\medskip

Now let $y$ be the neighbor of $x$ in $\overline{T}_1$ (see
Fig.~\ref{fig:notFarAway}). Obviously one branch at $z$ is
represented by the subtree $S \cup \overline{T}_2 \cup \{y\}$ with
weight $|S| + |\overline{T}_2| - 1$. We distinguish two cases,
$|T_1| < |S| + |T_2| - 1$ and $|T_1| \geq |S| + |T_2| - 1$.
\medskip

\textbf{Case A.}  In case $|T_1| < |S| + |T_2| - 1$, we get 
\begin{align*}
 \W_{\overline{T}}(y) = |S| + |\overline{T}_2| - 1 \geq |T_1| = \W_{\overline{T}}(x)
\end{align*}

\medskip

\textbf{Case B.} If $|T_1| \geq |S| + |T_2| - 1$, we get
\begin{align}
\W_{\overline{T}}(y) \leq |T_1| - 1 <  \W_{\overline{T}}(x),\label{eq:yIsInCentroid}
\end{align}
that is, $x$ is no centroid vertex of $\bar T$.
\input{transformation_4}

It remains to prove that $y$ is a centroid vertex of $\overline{T}$. For this we use inequality \eqref{eq:yIsInCentroid} and Lemma \ref{centroid} and we get
\begin{align*}
 \W_{\overline{T}}(y) < \W_{\overline{T}}(x) \leq \frac{n}{2} + 1
\end{align*}
which completes the proof.
\end{proof}

Proposition~\ref{transformationCentroid1} does not cover the case
$u = x$. So, if we want to replace the edge connecting the leaf
$b$ by an edge which is incident to the centroid vertex $x$ we
have the following result:

\begin{Pro}\label{transformationCentroid2}
Let $T = (V,E)$ be a tree and let further $x$ be a centroid vertex, $b$ a leaf
and $\bar b$ the vertex adjacent to $b$. Let further $\overline{T}
= (V , (E \setminus \{(\bar b,b)\}) \cup \{(x,b)\})$ be the tree
resulting from $T$ by removing the edge $(\bar b,b)$ and inserting
the edge $(x,b)$ instead. Then $x$ is a centroid vertex of
$\overline{T}$.
\end{Pro}

\begin{proof}
If $b$ is already adjacent to $x$ in $T$, then clearly $\overline{T} = T$.
Therefore vertex $x$ remains a centroid vertex. 
Otherwise, there is one
more branch at $x$ in $\overline{T}$ than in $T$ (see Fig.~\ref{fig:newBranch}).
The weight of the new branch is equal to 1  and, since $|T| = |\overline{T}| = n$, the branch weights of $x$ in $T$ and $\overline{T}$ satisfy 
\begin{align*}
 \W_T(x)-1 \leq \W_{\overline{T}}(x)\leq \W_T(x)\leq \frac{n}{2}
\end{align*}
which implies that $x$ is a centroid vertex of $\overline{T}$.

\input{transformation_5}

\end{proof}

%% file: transformation_2.tex
\begin{figure}[!h]
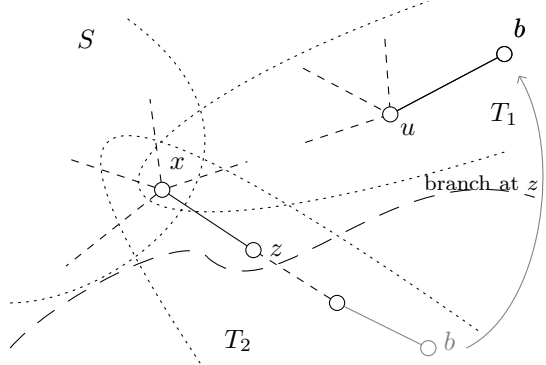

  \begin{center}
    \begin{\pic}{0.2cm}{0.2cm}{7.5cm}{5cm}
      \begin{pgfmagnify}{1}{1}  
        \pgfnodecircle{Node1}[stroke]{\pgfxy(2.5,2.5)}{0.1cm}
        \pgfnodecircle{Node2}[stroke]{\pgfxy(5.5,3.5)}{0.1cm}
        \pgfnodecircle{Node4}[virtual]{\pgfxy(2.3,4)}{0.3cm}
        \pgfnodecircle{Node5}[virtual]{\pgfxy(1,3)}{0.3cm}
        \pgfnodecircle{Node6}[virtual]{\pgfxy(1,1.3)}{0.3cm}
        \pgfnodecircle{Node7}[stroke]{\pgfxy(4.8,1)}{0.1cm}
        \pgfnodecircle{Node9}[stroke]{\pgfxy(3.7,1.7)}{0.1cm}
        \pgfnodecircle{Node10}[virtual]{\pgfxy(4,3)}{0.4cm}
        \pgfnodecircle{Node11}[virtual]{\pgfxy(5.4,4.9)}{0.4cm}
        \pgfnodecircle{Node12}[virtual]{\pgfxy(4,4.3)}{0.4cm}
        \pgfputat{\pgfnodeborder{Node1}{60}{0.3cm}}{\pgfbox[center,center]{$x$}}
\pgfputat{\pgfnodeborder{Node2}{320}{0.2cm}}{\pgfbox[center,center]{$u$}}
        \pgfputat{\pgfnodeborder{Node9}{0}{0.2cm}}{\pgfbox[center,center]{$z$}}
        \textcolor{gray}{
        \pgfnodecircle{Node3}[stroke]{\pgfxy(6,0.4)}{0.1cm}
        \pgfputat{\pgfnodeborder{Node3}{20}{0.2cm}}{\pgfbox[center,center]{$b$}}
        \pgfnodeconnline{Node7}{Node3}
        \pgfsetendarrow{\pgfarrowpointed}
        \pgfxycurve(6.5,0.4)(7.5,0.9)(7.8,3)(7.2,4)
        }
        \pgfsetdash{{1pt}{0pt}}{0pt}
      \textcolor{black}{
        \pgfnodecircle{Node8}[stroke]{\pgfxy(7,4.3)}{0.1cm}
        \pgfputat{\pgfnodeborder{Node8}{60}{0.3cm}}{\pgfbox[center,center]{$b$}}
        \pgfnodeconnline{Node2}{Node8}
        \pgfnodecircle{Node8}[stroke]{\pgfxy(7,4.3)}{0.1cm}
        \pgfputat{\pgfnodeborder{Node8}{60}{0.3cm}}{\pgfbox[center,center]{$b$}}
        \pgfnodeconnline{Node2}{Node8}
      }
        \pgfnodeconnline{Node1}{Node9}
      \pgfsetdash{{3pt}{3pt}}{0pt}
      \pgfnodeconnline{Node1}{Node4}
      \pgfnodeconnline{Node1}{Node5}
      \pgfnodeconnline{Node1}{Node6}
      \pgfnodeconnline{Node1}{Node10}
      \pgfnodeconnline{Node7}{Node9}
      \pgfnodeconnline{Node10}{Node2}
      \pgfnodeconnline{Node11}{Node2}
      \pgfnodeconnline{Node12}{Node2}
      \pgfsetdash{{3pt}{0pt}}{0pt}
      \pgfsetdash{{1pt}{2pt}}{0pt}
      \pgfxycurve(7,3)(0.5,1)(1,3.2)(6,5)
      \pgfputat{\pgfxy(7,3.5)}{\pgfbox[center,center]{$T_1$}}
      \pgfxycurve(3,0.2)(0.5,4)(1.5,4.3)(6.7,0.6)
      \pgfputat{\pgfxy(3.5,0.5)}{\pgfbox[center,center]{$T_2$}}
      \pgfxycurve(0.5,1)(2,1)(4.5 ,3)(2,4.8)
      \pgfputat{\pgfxy(1.5,4.5)}{\pgfbox[center,center]{$S$}}

      \pgfsetdash{{8pt}{6pt}}{0pt}
    \textcolor{black}{
      \pgfxycurve(0.5,0.4)(1,1)(2.8 ,2)(3.1,1.6)
      \pgfxycurve(3.1,1.6)(4,0.8)(5.5 ,3)(7.4,2.4)
      \pgfputat{\pgfxy(6.7,2.6)}{\pgfbox[center,center]{\footnotesize branch at
      $z$}}
      }
  \end{pgfmagnify}
 \end{\pic}
\caption{Neighbor $z$ of $x$ in $T_2$.}
\label{fig:notInT2}
 \end{center}
\end{figure}

%% file: transformation_3.tex
\begin{figure}[!ht]
  \begin{center}
    \begin{\pic}{0.2cm}{0.1cm}{7.6cm}{5.7cm}
       \begin{pgfmagnify}{1}{1}
      \pgfnodecircle{Node1}[stroke]{\pgfxy(2.5,2.5)}{0.1cm}
      \pgfnodecircle{Node2}[stroke]{\pgfxy(5.5,3.5)}{0.1cm}
      \pgfnodecircle{Node4}[virtual]{\pgfxy(2.3,4)}{0.3cm}
      \pgfnodecircle{Node5}[stroke]{\pgfxy(1,3)}{0.1cm}
      \pgfnodecircle{Node6}[virtual]{\pgfxy(1,1.3)}{0.3cm}
      \pgfnodecircle{Node7}[stroke]{\pgfxy(4.8,1)}{0.1cm}
      \pgfnodecircle{Node9}[virtual]{\pgfxy(3.7,1.7)}{0.4cm}
      \pgfnodecircle{Node10}[virtual]{\pgfxy(4,3)}{0.4cm}
      \pgfnodecircle{Node11}[virtual]{\pgfxy(5.4,4.9)}{0.4cm}
      \pgfnodecircle{Node12}[virtual]{\pgfxy(4,4.3)}{0.4cm}
      \pgfnodecircle{Node13}[virtual]{\pgfxy(0.5,4.3)}{0.4cm}
      \pgfnodecircle{Node14}[virtual]{\pgfxy(0.2,2)}{0.4cm}

  \pgfputat{\pgfnodeborder{Node1}{60}{0.3cm}}{\pgfbox[center,center]{$x$}}
  \pgfputat{\pgfnodeborder{Node2}{320}{0.2cm}}{\pgfbox[center,center]{$u$}}
  \pgfputat{\pgfnodeborder{Node5}{80}{0.2cm}}{\pgfbox[center,center]{$z$}}

\textcolor{gray}{
\pgfnodecircle{Node3}[stroke]{\pgfxy(6,0.4)}{0.1cm}
\pgfputat{\pgfnodeborder{Node3}{20}{0.2cm}}{\pgfbox[center,center]{$b$}}
\pgfnodeconnline{Node7}{Node3}
\pgfsetendarrow{\pgfarrowpointed}
\pgfxycurve(6.5,0.4)(7.5,0.9)(7.8,3)(7.2,4)
}
\pgfsetdash{{2pt}{0pt}}{0pt}

  \textcolor{black}{
  \pgfnodecircle{Node8}[stroke]{\pgfxy(7,4.3)}{0.1cm}
  \pgfputat{\pgfnodeborder{Node8}{60}{0.3cm}}{\pgfbox[center,center]{$b$}}
  \pgfnodeconnline{Node2}{Node8}
  \pgfnodecircle{Node8}[stroke]{\pgfxy(7,4.3)}{0.1cm}
  \pgfputat{\pgfnodeborder{Node8}{60}{0.3cm}}{\pgfbox[center,center]{$b$}}
  \pgfnodeconnline{Node2}{Node8}
}
  \pgfnodeconnline{Node1}{Node5}

  \pgfsetdash{{3pt}{3pt}}{0pt}
  \pgfnodeconnline{Node1}{Node4}
  \pgfnodeconnline{Node1}{Node10}
  \pgfnodeconnline{Node1}{Node6}
  \pgfnodeconnline{Node1}{Node9}
  \pgfnodeconnline{Node7}{Node9}
  \pgfnodeconnline{Node10}{Node2}
  \pgfnodeconnline{Node11}{Node2}
  \pgfnodeconnline{Node12}{Node2}
  \pgfnodeconnline{Node13}{Node5}
  \pgfnodeconnline{Node14}{Node5}

  \pgfsetdash{{3pt}{0pt}}{0pt}
  \pgfsetdash{{1pt}{2pt}}{0pt}
  \pgfxycurve(7,3)(0.5,1)(1,3.2)(6,5)
  \pgfputat{\pgfxy(7,3.5)}{\pgfbox[center,center]{$T_1$}}
  \pgfxycurve(3,0.2)(0.5,4)(1.5,4.3)(6.7,0.6)
  \pgfputat{\pgfxy(3.5,0.5)}{\pgfbox[center,center]{$T_2$}}
  \pgfxycurve(0.5,1)(2,1)(4.5 ,3)(2,5)
  \pgfputat{\pgfxy(1.5,4.5)}{\pgfbox[center,center]{$S$}}
  \pgfsetdash{{5pt}{1pt}}{0pt}

\textcolor{black}{
  \pgfsetdash{{8pt}{6pt}}{0pt}
  \pgfxycurve(0.5,0.1)(0.3,4.3)(1,4.3)(6,5.5)
  \pgfputat{\pgfxy(5.9,5.1)}{\pgfbox[center,center]{\footnotesize branch at $z$}}
}
\end{pgfmagnify}
\end{\pic}
\caption{Neighbor of $x$ in  $S$}
\label{fig:notInS}

 \end{center}
\end{figure}

%% file: transformation_4.tex
\begin{figure}[!h]
  \begin{center}
    \begin{\pic}{0.4cm}{0.1cm}{7.5cm}{5cm}
       \begin{pgfmagnify}{1}{1}
      \pgfnodecircle{Node1}[stroke]{\pgfxy(2.5,2.5)}{0.1cm}
      \pgfnodecircle{Node2}[stroke]{\pgfxy(5.5,3.5)}{0.1cm}
      \pgfnodecircle{Node4}[virtual]{\pgfxy(2.3,4)}{0.3cm}
      \pgfnodecircle{Node5}[virtual]{\pgfxy(1,3)}{0.1cm}
      \pgfnodecircle{Node6}[virtual]{\pgfxy(1,1.3)}{0.3cm}
      \pgfnodecircle{Node7}[stroke]{\pgfxy(4.8,1)}{0.1cm}
      \pgfnodecircle{Node9}[virtual]{\pgfxy(3.7,1.7)}{0.4cm}
      \pgfnodecircle{Node10}[stroke]{\pgfxy(4,3)}{0.1cm}
      \pgfnodecircle{Node11}[virtual]{\pgfxy(5.4,4.9)}{0.4cm}
      \pgfnodecircle{Node12}[virtual]{\pgfxy(4,4.3)}{0.4cm}
      \pgfnodecircle{Node13}[stroke]{\pgfxy(5,2.6)}{0.1cm}
      \pgfnodecircle{Node14}[virtual]{\pgfxy(0.2,2)}{0.4cm}

  \pgfputat{\pgfnodeborder{Node1}{60}{0.3cm}}{\pgfbox[center,center]{$x$}}
  \pgfputat{\pgfnodeborder{Node2}{320}{0.2cm}}{\pgfbox[center,center]{$u$}}
  \pgfputat{\pgfnodeborder{Node10}{80}{0.2cm}}{\pgfbox[center,center]{$y$}}
  \pgfputat{\pgfnodeborder{Node13}{80}{0.2cm}}{\pgfbox[center,center]{$z$}}

\textcolor{gray}{
\pgfnodecircle{Node3}[stroke]{\pgfxy(6,0.4)}{0.1cm}
\pgfputat{\pgfnodeborder{Node3}{20}{0.2cm}}{\pgfbox[center,center]{$b$}}
\pgfnodeconnline{Node7}{Node3}
\pgfsetendarrow{\pgfarrowpointed}
\pgfxycurve(6.5,0.4)(7.5,0.9)(7.8,3)(7.2,4)
}
\pgfsetdash{{2pt}{0pt}}{0pt}

  \textcolor{black}{
  \pgfnodecircle{Node8}[stroke]{\pgfxy(7,4.3)}{0.1cm}
  \pgfputat{\pgfnodeborder{Node8}{60}{0.3cm}}{\pgfbox[center,center]{$b$}}
  \pgfnodeconnline{Node2}{Node8}
  \pgfnodecircle{Node8}[stroke]{\pgfxy(7,4.3)}{0.1cm}
  \pgfputat{\pgfnodeborder{Node8}{60}{0.3cm}}{\pgfbox[center,center]{$b$}}
  \pgfnodeconnline{Node2}{Node8}
}
  \pgfnodeconnline{Node1}{Node10}
  \pgfnodeconnline{Node13}{Node10}

  \pgfsetdash{{3pt}{3pt}}{0pt}
  \pgfnodeconnline{Node1}{Node4}
  \pgfnodeconnline{Node1}{Node5}
  \pgfnodeconnline{Node1}{Node6}
  \pgfnodeconnline{Node1}{Node9}
  \pgfnodeconnline{Node7}{Node9}
  \pgfnodeconnline{Node10}{Node2}
  \pgfnodeconnline{Node11}{Node2}
  \pgfnodeconnline{Node12}{Node2}
  
  \pgfsetdash{{3pt}{0pt}}{0pt}

  \pgfsetdash{{1pt}{2pt}}{0pt}
  \pgfxycurve(7,2)(0.5,1)(1,3.2)(6,5)
  \pgfputat{\pgfxy(7,3.5)}{\pgfbox[center,center]{$T_1$}}
  \pgfxycurve(3,0.2)(0.5,4)(1.5,4.3)(6.7,0.6)
  \pgfputat{\pgfxy(3.5,0.5)}{\pgfbox[center,center]{$T_2$}}
  \pgfxycurve(0.5,1)(2,1)(4.5 ,3)(2,4.8)
  \pgfputat{\pgfxy(1.5,4.5)}{\pgfbox[center,center]{$S$}}
  \pgfsetdash{{5pt}{1pt}}{0pt}

\textcolor{black}{
  \pgfsetdash{{8pt}{6pt}}{0pt}
  \pgfxycurve(5,2)(2.3,3.6)(3.5,4.3)(7,2.2)
  \pgfputat{\pgfxy(5.5,2.3)}{\pgfbox[center,center]{$T'$}}
}
   \end{pgfmagnify}
   \end{\pic}
\caption{Distance from centroid vertex and $x$ is at most 1} 
\label{fig:notFarAway}

 \end{center}
\end{figure}

%% file: transformation_5.tex
\begin{figure}[!h]
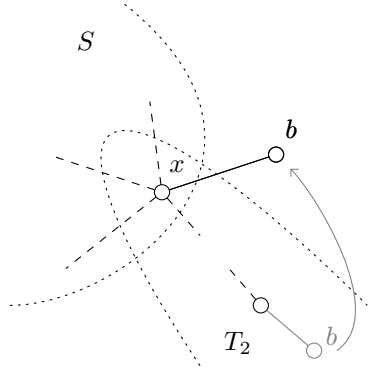

  \begin{center}
    \begin{\pic}{0.4cm}{0.1cm}{5.2cm}{5cm}
       \begin{pgfmagnify}{1}{1}
  
      \pgfnodecircle{Node1}[stroke]{\pgfxy(2.5,2.5)}{0.1cm}
      \pgfnodecircle{Node4}[virtual]{\pgfxy(2.3,4)}{0.3cm}
      \pgfnodecircle{Node5}[virtual]{\pgfxy(1,3)}{0.1cm}
      \pgfnodecircle{Node6}[virtual]{\pgfxy(1,1.3)}{0.3cm}
      \pgfnodecircle{Node7}[stroke]{\pgfxy(3.8,1)}{0.1cm}
      \pgfnodecircle{Node9}[virtual]{\pgfxy(3.2,1.7)}{0.3cm}
      \pgfnodecircle{Node11}[virtual]{\pgfxy(5.4,4.9)}{0.4cm}
      \pgfnodecircle{Node12}[virtual]{\pgfxy(4,4.3)}{0.4cm}

  \pgfputat{\pgfnodeborder{Node1}{60}{0.3cm}}{\pgfbox[center,center]{$x$}}

\textcolor{gray}{
\pgfnodecircle{Node3}[stroke]{\pgfxy(4.5 ,0.4)}{0.1cm}
\pgfputat{\pgfnodeborder{Node3}{40}{0.2cm}}{\pgfbox[center,center]{$b$}}
\pgfnodeconnline{Node7}{Node3}
\pgfsetendarrow{\pgfarrowpointed}
\pgfxycurve(4.8,0.4)(5.5,0.9)(4.5,2.5)(4.2,2.8)
}
\pgfsetdash{{2pt}{0pt}}{0pt}

\textcolor{black}{
  \pgfnodecircle{Node8}[stroke]{\pgfxy(4,3)}{0.1cm}
  \pgfputat{\pgfnodeborder{Node8}{60}{0.3cm}}{\pgfbox[center,center]{$b$}}
  \pgfnodeconnline{Node1}{Node8}
  \pgfnodecircle{Node8}[stroke]{\pgfxy(4,3)}{0.1cm}
  \pgfputat{\pgfnodeborder{Node8}{60}{0.3cm}}{\pgfbox[center,center]{$b$}}
  \pgfnodeconnline{Node1}{Node8}
}

  \pgfsetdash{{3pt}{3pt}}{0pt}
  \pgfnodeconnline{Node1}{Node4}
  \pgfnodeconnline{Node1}{Node5}
  \pgfnodeconnline{Node1}{Node6}
  \pgfnodeconnline{Node1}{Node9}
  \pgfnodeconnline{Node7}{Node9}
  
  \pgfsetdash{{1pt}{1pt}}{0pt}
  \pgfsetdash{{1pt}{2pt}}{0pt}
  \pgfxycurve(3,0.2)(0.5,4)(1.8,4.4)(5.2,1)
  \pgfputat{\pgfxy(3.5,0.5)}{\pgfbox[center,center]{$T_2$}}
  \pgfxycurve(0.5,1)(2,1)(4.5 ,3)(2,5)
  \pgfputat{\pgfxy(1.5,4.5)}{\pgfbox[center,center]{$S$}}
\end{pgfmagnify}
\end{\pic}
\caption{New branch at vertex $x$}
\label{fig:newBranch}

 \end{center}
\end{figure}

%% file: lowerBound.tex
\section{Lower Bounds}
\label{sec:lower}

In this section we derive lower bounds for the status and the radius of
trees. In particular, we shall consider so-called
\textit{$k$-balanced} trees. A tree $T$ of maximum degree $k$ is
called $k$-balanced if there exists a vertex $x$ such that for every 
vertex $z$ with $\dist(z,x) \leq \rad(x) - 2$, $\deg(z) = k$ holds. 
In general there are non-isomorphic $k$-balanced trees of order $n$, 
but it is rather obvious that both the status and the radius of a 
balanced tree depends only on $k$ and $n$. Therefore, we are not interested 
in a certain instance and denote an arbitrary 
$k$-balanced tree of order $n$ by $B_{n,k}$.

The following theorems show that both the status
and the radius of a tree of order $n$ and maximum degree $k$ are
bounded by the status and radius of  $k$-balanced trees of order
$n$. First we state the assertion on the radius.

\begin{The}\label{lowerboundexc}
Let $T$ be a tree of order $n$ and of  maximum degree $\Delta(T) = k$. 
Its radius is greater than or equal to the radius of a 
$k$-balanced tree of order $n$, that is,
\begin{align*}
 \rad(B_{n,k}) \leq \rad(T).
\end{align*}
\end{The}

\begin{proof}
Let $x$ be a central vertex of $T$ and $r= \rad(T)$. If $T$ is not
$k$-balanced, there exists a vertex $u$ such that $\deg_T(u) < k$
and $d_T(x,u) \leq r-2$. Further let $b$ be an endvertex of a
longest path. We can choose $b$ such
that $\Delta(T\setminus \{b\}) = k$. Now let $\T$ be the tree
which results by removing the edge connecting $b$ and inserting
the edge $(u,b)$ instead. Obviously $\T$ is a tree of order $n$
and maximum degree $k$.

If $T$ has just one longest path, the transformation from $T$ to
$\T$ shortens the longest path by 1. If there are several longest
paths, the radius remains the same. Overall we can conclude that
$\rad(\T) \leq \rad(T)$ holds in every case.
\end{proof}

As already mentioned the status behaves similarly. In fact, in
case of the status the bounds are even strict, that is,
$s(T)=s(B_{n,k})$ if and only if $T$ is $k$-balanced, cf. Lin et al.~in
\cite{LinShangZhang11}. We will give here an alternate proof for
this statement using the method described in the Section 2.
The proof is similar to the one of Theorem~\ref{lowerboundexc}, only
the choice of the leaf used for the transformation needs to be
done more carefully.

\begin{The}(Lin et al., 2011, \cite{LinShangZhang11})\label{lowerboundstatus}
Let $T$ be a tree of order $n$ and maximum degree $\Delta(T) = k$.
Its status is greater than or equal to the status of a $k$-balanced tree
of order $n$, that is,
\begin{equation*}
 \s(B_{n,k}) \leq \s(T)
\end{equation*}
with equality if and only if $T = B_{n,k}$.
\end{The}

\begin{proof}
In case $T$ is not a $k$-balanced tree  we can
transform $T$ to a tree $\T$ of same order and maximum degree such
that $\s(\T) < \s(T)$. Let $x$ be a centroid vertex of $T$ and $r
= \ecc(x)$ its eccentricity. We split into the two cases
\begin{enumerate}
 \item[(1)] $\deg(x) = k$ and
 \item[(2)] $\deg(x) < k$.
\end{enumerate}

First, let us consider the case $\deg(x) = k$. Since $T$ is
 not $k$-balanced, we know that $r>1$ and  that there exists a vertex $u$ such
that $\dist_T(x,u) \leq r-2$ and $\deg_T(u) < k$. Choose a leaf
$b$ with $\dist_T(x,b) = r$. Every
choice of $b$ yields $\Delta(T\setminus\{b\}) = k$. Let $T_1$ be
the branch at $x$ containing $u$, $T_2$ the branch containing $b$,
$S = (T \setminus(T_1 \cup T_2)) \cup \{x\}$ the union of all
other branches and $\T$ be the tree which results by removing the
edge connecting $b$ and inserting the edge $(u,b)$ instead. Since
$\deg(x) = k$ we know that $x \neq u$ and therefore we can apply
Proposition \ref{transformationCentroid1}. Thus we know, that $x$
is a centroid vertex of $\bar T$ if $T_1 = T_2$ or $|T_1| < |T_2| + |S| -
1$. In this case we obtain the following.
\begin{align*}
 \s(\T) &= \s(x) = \sum_{z \in \T} \dist_{\T}(x,z) \\
        &= \sum_{z \in T \setminus \{b\}} \dist_T(x,z) + \dist_{\T}(x,b) \\
        &= \s(T) - \dist_T(x,b) + \dist_T(x,u) + 1 \\
        &< \s(T).
\end{align*}
\input{transformation_6}

In case $T_1 \neq T_2$ and $|T_1| \geq |T_2| + |S| - 1$ we know
that the neighbour $y$ of $x$ on the shortest path from $x$ to $u$
is a centroid vertex (see Fig. \ref{fig:centroidIsMoving}).
\noindent We get
\begin{align*}
 \s(\T) &= \s(y) = \sum_{z \in \T} \dist_{\T}(y,z) \\
        &= \sum_{z \in (T_1 \cup \{b\}) \setminus \{x\}} \dist_{\T}(y,z) + \sum_{z \in S \cup (T_2 \setminus \{b\}) } \dist_{\T}(y,z) \\
        &= \sum_{z \in T_1\setminus \{x\}} (\dist_T(x,z) - 1) + \dist_{\T}(y,b) + \sum_{z \in S \cup (T_2 \setminus \{b\})} (\dist_T(x,z) + 1) \\
        &= \s(T) - \dist_T(x,b) - (|T_1|- 1) + \dist_T(x,u) + |S| + |T_2| - 2 \\
        &= \s(T)  + (\dist_T(x,u) - \dist_T(x,b)) + (|S| + |T_2| -|T_1|- 1)\\
        &< \s(T).
\end{align*}
Thus case $\deg(x) = k$ is complete.

Henceforth let us consider the case $\deg_T(x) < k$. Since $T$ is
not $k$-balanced we can choose a leaf $b$ such that $\dist_T(x,b)>1$ 
and $\Delta(T\setminus\{b\}) = k$. 
(Otherwise $x$ would have $\deg(x) - 1$ branches which consist of a single leaf 
and one branch with a vertex of degree $k$, which contradicts the assumption that 
$x$ is a centroid vertex.) Let $\T$ be the tree which results
by removing the edge connecting $b$ and inserting edge $(x,b)$.
According to Proposition \ref{transformationCentroid2} $x$ is a
centroid vertex of  $\T$ and we obtain
\begin{align*}
 \s(\T) = \s(T) - \dist_T(x,b) + 1
        &< \s(T).
\end{align*}
\end{proof}

%% file: transformation_6.tex
\begin{figure}[!h]
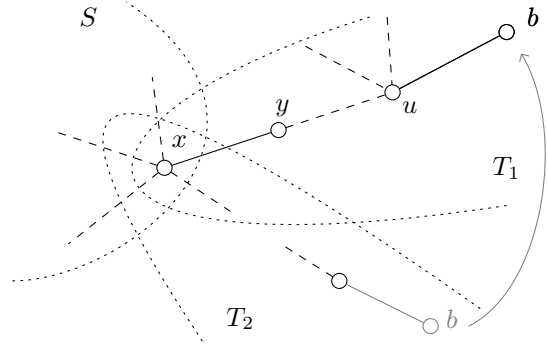

  \centering
    \begin{\pic}{0.4cm}{0.1cm}{7.5cm}{4.7cm}
     \begin{pgfmagnify}{1}{1}
      \pgfnodecircle{Node1}[stroke]{\pgfxy(2.5,2.5)}{0.1cm}
      \pgfnodecircle{Node2}[stroke]{\pgfxy(5.5,3.5)}{0.1cm}
      \pgfnodecircle{Node4}[virtual]{\pgfxy(2.3,4)}{0.3cm}
      \pgfnodecircle{Node5}[virtual]{\pgfxy(1,3)}{0.1cm}
      \pgfnodecircle{Node6}[virtual]{\pgfxy(1,1.3)}{0.3cm}
      \pgfnodecircle{Node7}[stroke]{\pgfxy(4.8,1)}{0.1cm}
      \pgfnodecircle{Node9}[virtual]{\pgfxy(3.7,1.7)}{0.4cm}
      \pgfnodecircle{Node10}[stroke]{\pgfxy(4,3)}{0.1cm}
      \pgfnodecircle{Node11}[virtual]{\pgfxy(5.4,4.9)}{0.4cm}
      \pgfnodecircle{Node12}[virtual]{\pgfxy(4,4.3)}{0.4cm}
      \pgfnodecircle{Node14}[virtual]{\pgfxy(0.2,2)}{0.4cm}
      \pgfputat{\pgfnodeborder{Node1}{60}{0.3cm}}{\pgfbox[center,center]{$x$}}
      \pgfputat{\pgfnodeborder{Node2}{320}{0.2cm}}{\pgfbox[center,center]{$u$}}
      \pgfputat{\pgfnodeborder{Node10}{80}{0.2cm}}{\pgfbox[center,center]{$y$}}
      \textcolor{gray}{
        \pgfnodecircle{Node3}[stroke]{\pgfxy(6,0.4)}{0.1cm}
        \pgfputat{\pgfnodeborder{Node3}{20}{0.2cm}}{\pgfbox[center,center]{$b$}}
         \pgfnodeconnline{Node7}{Node3}
        \pgfsetendarrow{\pgfarrowpointed}
        \pgfxycurve(6.5,0.4)(7.5,0.9)(7.8,3)(7.2,4)
      }
      \pgfsetdash{{2pt}{0pt}}{0pt}
      \textcolor{black}{
        \pgfnodecircle{Node8}[stroke]{\pgfxy(7,4.3)}{0.1cm}
        \pgfputat{\pgfnodeborder{Node8}{30}{0.3cm}}{\pgfbox[center,center]{$b$}}
        \pgfnodeconnline{Node2}{Node8}
        \pgfnodecircle{Node8}[stroke]{\pgfxy(7,4.3)}{0.1cm}
        \pgfputat{\pgfnodeborder{Node8}{30}{0.3cm}}{\pgfbox[center,center]{$b$}}
        \pgfnodeconnline{Node2}{Node8}
      }
      \pgfnodeconnline{Node1}{Node10}
      \pgfsetdash{{3pt}{3pt}}{0pt}
      \pgfnodeconnline{Node1}{Node4}
      \pgfnodeconnline{Node1}{Node5}
      \pgfnodeconnline{Node1}{Node6}
      \pgfnodeconnline{Node1}{Node9}
      \pgfnodeconnline{Node7}{Node9}
      \pgfnodeconnline{Node10}{Node2}
      \pgfnodeconnline{Node11}{Node2}
      \pgfnodeconnline{Node12}{Node2}
      \pgfsetdash{{3pt}{0pt}}{0pt}
        \pgfsetdash{{1pt}{2pt}}{0pt}
        \pgfxycurve(7,2)(0.5,1)(1,3.2)(5.2,4.5)
        \pgfputat{\pgfxy(7,2.5)}{\pgfbox[center,center]{$T_1$}}
        \pgfxycurve(3,0.2)(0.5,4)(1.5,4.3)(6.7,0.6)
        \pgfputat{\pgfxy(3.5,0.5)}{\pgfbox[center,center]{$T_2$}}
      \pgfxycurve(0.5,1)(2,1)(4.5 ,3)(2,4.7)
      \pgfputat{\pgfxy(1.5,4.5)}{\pgfbox[center,center]{$S$}}
     \end{pgfmagnify}
    \end{\pic}
    \caption{$y$ is a centroid vertex after transformation}
    \label{fig:centroidIsMoving}
\end{figure}

%% file: upperBound.tex
\section{Upper Bounds}
\label{sec:upper}

In this section we discuss upper bounds for the radius and status
of trees. First we introduce star-like trees, so-called
\textit{$k$-comets}. 
Let $n\ge 3$ and the maximum degree $k \ge 2$. We say a tree of order
$n$ is a  \textit{$k$-comet} $C_{n,k}$ if there exists a vertex $v$ 
of degree $k$ which is lying on a path of  length $n-k+1$.
A $k$-comet is called $S_{n,k}$ if $k-1$ neighbors of $v$ are leaves. 

But there is also another star-like tree we need to consider, it is a tree which is 
``almost'' a $C_{n,k}$. More precisely, we denote a tree of order $n$ and maximum degree $k$ 
by $C^{\star}_{n,k}$ if it contains a $C_{n-1,k}$ as a subtree. 
For examples see Figures~\ref{fig:comets}. 
A tree $S_{n,k}$ is uniquely defined whereas the trees $C_{n,k}$ and $C^{\star}_{n,k}$ may have different topologies.

\input{Kometen}

It will be shown that both radius and status take their maximum
value on these star-like trees. First we present the statement
for the radius. Note that the upper bound of the radius was already 
shown by Vizing \cite[Lemma 1]{Viz67}.

\begin{The}\label{upperboundexc}
Let $T$ be a tree of order $n$ and maximum  degree $\Delta(T) = k$. Then its radius is less than or equal to $\lceil
(n-k+1)/2\rceil$.
For equality we get in particular the following.
\begin{enumerate}
 \item[(i)] If $n-k+1$ is even, then $\rad(T) =\frac{n-k+1}{2}$ if and
       only if  $T=C^{\star}_{n,k}$.
 \item[(ii)] If $n-k+1$ is odd, then $\rad(T) = \frac{n-k}{2} +1 $ if and
       only if  $T=C_{n,k}$.
\end{enumerate}

\end{The}

\begin{proof}
The proof of this assertion is straight-forward counting exercise. Let $\ell$ be the length of the longest path. We need to have at least one vertex of degree $k$, so $\ell +k-2$ edges are in use. Clearly $\ell +k-2 \leq n-1$, which implies that $\ell \leq n-k+1$.

Now assume $T$ is a tree with $\rad(T)= \lceil \frac{n-k+1}{2}\rceil$. If $n-k+1$ is odd, there is a path of length $n-k+1$ and therefore all $k-2$ edges which are not on this path need to be incident to one and the same vertex. Hence $T$ is a $C_{n,k}$.
If $n-k+1$ is even, the longest path of $T$ has length at least $n-k$. This means that there is one edge whose position is not important, neither to the radius nor to the condition $\Delta(T)=k$. Deleting that edge would result in a $C_{n-1,k}$, so the given tree is a $C_{n,k}^{\star}$.
\end{proof}

\begin{Exa} The trees in Fig.~\ref{fig:CometSameRadius} and Fig.~\ref{fig:noCometButSameRadius}
are both of order $n= 14$ and maximal degree $k=9$. They have both minimum radius but the tree in Fig.~\ref{fig:noCometButSameRadius} is not a $k$-comet.

\end{Exa}

The next statement concerns the status.  Lin et al.~presented 
this result in \cite{LinShangZhang11}, but again we want to
present an alternate, simpler proof using tree transformations.

\begin{The}(Lin et al.~2011, \cite{LinShangZhang11})\label{upperboundstatus}
Given a tree $T$ of order $n$  and maximum degree $\Delta(T)=k$,
its status is less than or equal to the status of a $k$-comet,  that is,
\begin{equation*}
\s(T) \leq \s(S_{n,k})
\end{equation*}
 with equality if and only if $T = S_{n,k}$.
\end{The}

\begin{proof}
Since a path is a 2-comet we can assume $k>2$. Let $x$ be a
centroid vertex and let $z$ be a vertex with $\deg(z) = k$. If $T
\neq S_{n,k}$ then there exist two leaves $b$ and $\overline{b}$
with $\dist_T(z,b), \dist_T(z,\overline{b}) > 1$. Without loss of
generality let $\distT(x,b) \geq \distT(x,\overline{b})$ hold
during this proof.

Let $\overline{T}$ be the tree resulting from the removal of the
edge incident with $\overline{b}$ and inserting the edge
$(b,\overline{b})$ instead. Then $\overline{T}$ is again a tree of
order $n$ and maximum degree $k$.

If it is possible to choose $b$ and $\overline{b}$ such that $x$
is also a centroid vertex of $\overline{T}$, then the following
holds.

\begin{align}
 \s(\overline{T}) &= \s(x) = \sum_{u \in \overline{T}} \dist_{\overline{T}}(x,u) \notag\\
      &=\sum_{u \in T \setminus \{\overline{b}\}} \dist_T(x,u) + \dist_{\overline{T}}(x, \overline{b}) \label{xBleibtZentroid}\\
      &=\sum_{u \in T} \dist_T(x,u) - \dist_T(x,\overline{b}) + (\dist_T(x, b) + 1) \notag \\
      &> \s(T). \notag
\end{align}

Now we need to discuss how to choose $b$ and
$\overline{b}$, such that either $x$ is a centroid vertex of $\overline{T}$ or,
if this cannot be assured, we can guarantee $\s(T) >
\s(\overline{T})$ otherwise.

\input{SternTransformation_0}

We split into the two cases $x = z$ and $x \neq z$.
First, let $x \neq z$. Let $T_z$ be the branch
at $x$ containing $z$. If $T_z$ contains two leaves with distance
to $z$ greater than one, then choose $b$ and $\overline{b}$ in
$T_z$ (see Fig.~\ref{fig:SternTransformation_0}). According to
Proposition~\ref{transformationCentroid1} we know that $x$ is a
centroid vertex of $\overline{T}$ and together with
\eqref{xBleibtZentroid} follows $\s(T)<\s(\overline{T})$.

If there exists a branch $T_1 \neq T_z$ which is not a path, we
can choose $b$ and $\overline{b}$ inside this branch (see
Fig.~\ref{fig:SternTransformation_1}). According to Proposition
\ref{transformationCentroid1} $x$ is a centroid vertex of
$\overline{T}$ and therefore we again conclude with
\eqref{xBleibtZentroid} that $\s(T)<\s(\overline{T})$.

So let us assume all branches not equal to $T_z$ to be paths and
let $T_z$ contain at most one leaf with distance to $z$ greater
than one. 
Consider $T_b$ the branch containing
$b$, $T_{\overline{b}}$ the branch containing $\overline{b}$ and
$S = T \setminus (T_b \cup T_{\overline b}) \cup \{x\}$. Obviously
$T_b \neq T_{\overline b}$.

\noindent If there exists a choice of $b$ and $\overline{b}$ such that $|T_b| < |T_{\overline{b}}| + |S| - 1$ holds, then
$x$ is a centroid vertex of $\overline T$ according to Proposition
\ref{transformationCentroid1}. With \eqref{xBleibtZentroid} we get
$\s(T)<\s(\overline T)$.

\noindent Therefore, let us consider the case $|T_b| \geq
|T_{\overline{b}}| + |S| - 1$ for every possible choice of $b$ and $\overline{b}$. 
First of all, this assumption implies $\distT(x,b)>\distT(x,\overline{b})$. Assume equality holds. 
Then, we derive from our assumption (by exchanging the roles of $b$ and $\overline{b}$)  
$|T_{\overline{b}}| = |T_b|$ and $|S|=1$. Since one of the branches is a path and the leaves $b$ and $\overline{b}$ 
have the same distance to $x$, both branches must be paths. But these are the only branches at $x$, so we get a 
contradiction to $\Delta(T)=k>2$.

Proposition~\ref{transformationCentroid1} states that under the actual assumptions on the 
choice of $b$ and $\overline{b}$ the neighbor $y$ of
$x$ on the shortest path from $x$ to $b$ is a centroid vertex.
Thus we get the following.
\begin{align}
\s(\overline{T}) &= \s(y) = \sum_{u \in T} \dist_{\overline{T}}(y,u) \notag\\
       &= \sum_{u \in T_{b} \setminus \{x\}} \dist_{\overline{T}}(y,u) + \dist_{\overline{T}}(y,\overline{b})+ \sum_{u \in S \cup T_{\overline{b}} \setminus \{\overline{b}\} } \dist_{\overline{T}}(y,u) \label{eq:diff}\\
       &= \sum_{u \in T} \dist_T(x,u) - \dist_T(x,\overline{b}) - |T_b| + 1 + \dist_T(x,b) + |S| + |T_{\overline{b}}| - 2 \notag\\
       &= \s(T) + (\dist_T(x,b)- \dist_T(x,\overline{b})  ) + (|S| + |T_{\overline{b}}| - |T_{b}| - 1)\notag.
\end{align}

To get the desired result, that is, $\s(\bar T) > \s(T)$, we can use a lot of assumptions on $T$ we have at this point of the proof. 
We know that since $T$ is not a comet, there exist at least two leaves of $T$ whose distance to $z$ is greater than or equal to 1. If there are exactly two leaves $b$ and $\overline{b}$ with this property, then there are also exactly two branches at $x$, that is $|S|=1$ and one branch at $x$ is path and the other is $T_z$. Since inequality \eqref{eq:centroidInequs}  (and  $\distT(x,b)>\distT(x,\overline{b})$) must hold, $b$ is the unique leaf in the branch which is a path and $\overline{b}$ is the unique leaf in $T_z$ which has distance to $z$ greater than 1. Therefore 
$\distT(x,b) = |T_b|-1$ and $\distT(x,\overline{b}) = |T_{\overline{b}}|-(k-1)$ (and $|S|=1$) and we get (using equations \eqref{eq:diff})
\begin{align*}
 \s(\overline{T}) - \s(T) = k-2 > 0
\end{align*}

Now assume there exist more than two leaves which are suitable for the roles of $b$ and $\overline{b}$. As we observed above for every choice of $b$ and $\overline{b}$, $\distT(x,b)>\distT(x,\overline{b})$ holds, hence a third (suitable) leaf $\hat  b$ satisfies $\distT(x,\hat b) \neq \distT(x,b)$ and $\distT(x,\hat b)\neq \distT(x,\overline{b})$. Therefore we can choose the vertices $b$ and $\overline{b}$ with the property $\distT(x,b) > \distT(x,\overline{b}) + 1$. Using again the equalities \eqref{eq:diff} we get 
\begin{align*}
 \s(\overline{T}) - \s(T) > 0
\end{align*}
and the case $x \neq z$ is complete.

Let now $z=x$, that is, $\deg(x)=k$. In this case there is no branch at $x$ with a special role, 
we do not need to distinguish the branches at $x$. 
With the same arguments as in the first case if either one branch is not a path or if there exists 
a choice of $b$ and $\overline{b}$ such that (with the notation from above) 
$|T_b|< |S| +|T_{\overline{b}}| - 1$, $x$ is a centroid vertex of the transformed tree and we 
are done due to \eqref{xBleibtZentroid}.

Otherwise, we can use the equalities in \eqref{eq:diff}. Every branch at $x$ is a path, that means $\distT(x,b) = |T_b| - 1$ and $\distT(x,\overline{b}) = |T_{\overline{b}}| - 1$ and, since $k>2$ we know that $|S|>1$. Altogether we get
\begin{align*}
\s(\overline{T}) - \s(T)&=(|S| + |T_{\overline{b}}|- |T_b|-1) + (\dist_T(x,b) - \dist_T(x,\overline{b})) 
\\
&= |S|-1 > 0
\end{align*}
and the proof is complete.
\end{proof}

%% file: Kometen.tex
\begin{figure}[!h]
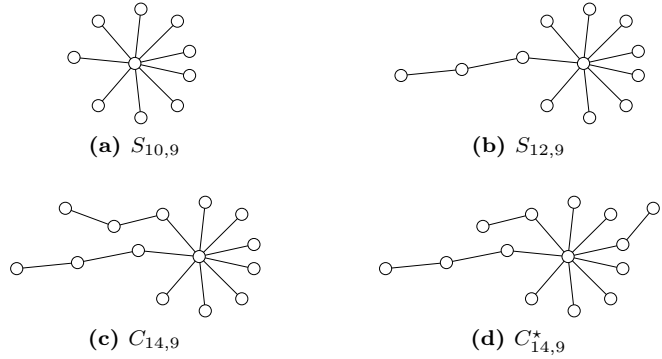


\centering
\subfloat[$S_{10,9}$]{ %
  \begin{\pic}{1.35cm}{0.8cm}{5.05cm}{2.5cm}
    \begin{pgfmagnify}{0.8}{0.8}
      \pgfnodecircle{Node1}[stroke]{\pgfxy(4,2)}{0.1cm}
      \pgfnodecircle{Node2}[stroke]{\pgfxy(4.1,2.9)}{0.1cm}
      \pgfnodecircle{Node3}[stroke]{\pgfxy( 4.7 ,2.7 )}{0.1cm}
      \pgfnodecircle{Node4}[stroke]{\pgfxy(4.9,2.2)}{0.1cm}
      \pgfnodecircle{Node5}[stroke]{\pgfxy(4.9,1.8)}{0.1cm}
      \pgfnodecircle{Node6}[stroke]{\pgfxy(4.7,1.3)}{0.1cm}
      \pgfnodecircle{Node7}[stroke]{\pgfxy(4.1,1.1)}{0.1cm}
      \pgfnodecircle{Node8}[stroke]{\pgfxy(3.4,1.3)}{0.1cm}
      \pgfnodecircle{Node9}[stroke]{\pgfxy(3.4,2.7)}{0.1cm}
      \pgfnodecircle{Node10}[stroke]{\pgfxy(3,2.1)}{0.1cm}
      \pgfnodeconnline{Node1}{Node2}
      \pgfnodeconnline{Node1}{Node3}
      \pgfnodeconnline{Node1}{Node4}
      \pgfnodeconnline{Node1}{Node5}
      \pgfnodeconnline{Node1}{Node6}
      \pgfnodeconnline{Node1}{Node7}
      \pgfnodeconnline{Node1}{Node8}
      \pgfnodeconnline{Node1}{Node9}
      \pgfnodeconnline{Node1}{Node10}
    \end{pgfmagnify}
  \end{\pic}
}
\hspace{1cm}
\subfloat[$S_{12,9}$]{ %
  \begin{\pic}{0.5cm}{0.8cm}{4.2cm}{2.5cm}
    \begin{pgfmagnify}{0.8}{0.8}
      \pgfnodecircle{Node1}[stroke]{\pgfxy(4,2)}{0.1cm}
      \pgfnodecircle{Node2}[stroke]{\pgfxy(4.1,2.9)}{0.1cm}
      \pgfnodecircle{Node3}[stroke]{\pgfxy( 4.7 ,2.7 )}{0.1cm}
      \pgfnodecircle{Node4}[stroke]{\pgfxy(4.9,2.2)}{0.1cm}
      \pgfnodecircle{Node5}[stroke]{\pgfxy(4.9,1.8)}{0.1cm}
      \pgfnodecircle{Node6}[stroke]{\pgfxy(4.7,1.3)}{0.1cm}
      \pgfnodecircle{Node7}[stroke]{\pgfxy(4.1,1.1)}{0.1cm}
      \pgfnodecircle{Node8}[stroke]{\pgfxy(3.4,1.3)}{0.1cm}
      \pgfnodecircle{Node9}[stroke]{\pgfxy(3.4,2.7)}{0.1cm}
      \pgfnodecircle{Node10}[stroke]{\pgfxy(3,2.1)}{0.1cm}
      \pgfnodecircle{Node11}[stroke]{\pgfxy(2,1.9)}{0.1cm}
      \pgfnodecircle{Node12}[stroke]{\pgfxy(1,1.8)}{0.1cm}
      \pgfnodeconnline{Node1}{Node2}
      \pgfnodeconnline{Node1}{Node3}
      \pgfnodeconnline{Node1}{Node4}
      \pgfnodeconnline{Node1}{Node5}
      \pgfnodeconnline{Node1}{Node6}
      \pgfnodeconnline{Node1}{Node7}
      \pgfnodeconnline{Node1}{Node8}
      \pgfnodeconnline{Node1}{Node9}
      \pgfnodeconnline{Node1}{Node10}
      \pgfnodeconnline{Node10}{Node11}
      \pgfnodeconnline{Node12}{Node11}
    \end{pgfmagnify}
  \end{\pic}
}

\subfloat[$C_{14,9}$]{ %
  \begin{\pic}{0.5cm}{0.8cm}{4.2cm}{2.5cm}
    \begin{pgfmagnify}{0.8}{0.8}
      \pgfnodecircle{Node1}[stroke]{\pgfxy(4,2)}{0.1cm}
      \pgfnodecircle{Node2}[stroke]{\pgfxy(4.1,2.9)}{0.1cm}
      \pgfnodecircle{Node3}[stroke]{\pgfxy( 4.7 ,2.7 )}{0.1cm}
      \pgfnodecircle{Node4}[stroke]{\pgfxy(4.9,2.2)}{0.1cm}
      \pgfnodecircle{Node5}[stroke]{\pgfxy(4.9,1.8)}{0.1cm}
      \pgfnodecircle{Node6}[stroke]{\pgfxy(4.7,1.3)}{0.1cm}
      \pgfnodecircle{Node7}[stroke]{\pgfxy(4.1,1.1)}{0.1cm}
      \pgfnodecircle{Node8}[stroke]{\pgfxy(3.4,1.3)}{0.1cm}
      \pgfnodecircle{Node9}[stroke]{\pgfxy(3.4,2.7)}{0.1cm}
      \pgfnodecircle{Node10}[stroke]{\pgfxy(3,2.1)}{0.1cm}
      \pgfnodecircle{Node11}[stroke]{\pgfxy(2,1.9)}{0.1cm}
      \pgfnodecircle{Node12}[stroke]{\pgfxy(1,1.8)}{0.1cm}
      \pgfnodecircle{Node13}[stroke]{\pgfxy(2.6,2.5)}{0.1cm}
      \pgfnodecircle{Node14}[stroke]{\pgfxy(1.8,2.8)}{0.1cm}
      \pgfnodeconnline{Node1}{Node2}
      \pgfnodeconnline{Node1}{Node3}
      \pgfnodeconnline{Node1}{Node4}
      \pgfnodeconnline{Node1}{Node5}
      \pgfnodeconnline{Node1}{Node6}
      \pgfnodeconnline{Node1}{Node7}
      \pgfnodeconnline{Node1}{Node8}
      \pgfnodeconnline{Node1}{Node9}
      \pgfnodeconnline{Node1}{Node10}
      \pgfnodeconnline{Node10}{Node11}
      \pgfnodeconnline{Node12}{Node11}
      \pgfnodeconnline{Node9}{Node13}
      \pgfnodeconnline{Node13}{Node14}
    \end{pgfmagnify}
   \label{fig:CometSameRadius}
  \end{\pic}
}
\hspace{1cm}
\subfloat[$C^{\star}_{14,9}$]{ %
    \begin{\pic}{0.7cm}{0.8cm}{4.4cm}{2.5cm}
      \begin{pgfmagnify}{0.8}{0.8}
        \pgfnodecircle{Node1}[stroke]{\pgfxy(4,2)}{0.1cm}
        \pgfnodecircle{Node2}[stroke]{\pgfxy(4.1,2.9)}{0.1cm}
        \pgfnodecircle{Node3}[stroke]{\pgfxy( 4.7 ,2.7 )}{0.1cm}
        \pgfnodecircle{Node4}[stroke]{\pgfxy(4.9,2.2)}{0.1cm}
        \pgfnodecircle{Node5}[stroke]{\pgfxy(4.9,1.8)}{0.1cm}
        \pgfnodecircle{Node6}[stroke]{\pgfxy(4.7,1.3)}{0.1cm}
        \pgfnodecircle{Node7}[stroke]{\pgfxy(4.1,1.1)}{0.1cm}
        \pgfnodecircle{Node8}[stroke]{\pgfxy(3.4,1.3)}{0.1cm}
        \pgfnodecircle{Node9}[stroke]{\pgfxy(3.4,2.7)}{0.1cm}
        \pgfnodecircle{Node10}[stroke]{\pgfxy(3,2.1)}{0.1cm}
        \pgfnodecircle{Node11}[stroke]{\pgfxy(2,1.9)}{0.1cm}
        \pgfnodecircle{Node12}[stroke]{\pgfxy(1,1.8)}{0.1cm}
        \pgfnodecircle{Node13}[stroke]{\pgfxy(2.6,2.5)}{0.1cm}
        \pgfnodecircle{Node14}[stroke]{\pgfxy(5.4,2.8)}{0.1cm}
        \pgfnodeconnline{Node1}{Node2}
        \pgfnodeconnline{Node1}{Node3}
        \pgfnodeconnline{Node1}{Node4}
        \pgfnodeconnline{Node1}{Node5}
        \pgfnodeconnline{Node1}{Node6}
        \pgfnodeconnline{Node1}{Node7}
        \pgfnodeconnline{Node1}{Node8}
        \pgfnodeconnline{Node1}{Node9}
        \pgfnodeconnline{Node1}{Node10}
        \pgfnodeconnline{Node10}{Node11}
        \pgfnodeconnline{Node12}{Node11}
        \pgfnodeconnline{Node9}{Node13}
        \pgfnodeconnline{Node4}{Node14}
      \end{pgfmagnify}
    \end{\pic}
    \label{fig:noCometButSameRadius}

}
\caption{$k$-comets}
    \label{fig:comets}
\end{figure}

%% file: SternTransformation_0.tex
\begin{figure}[!ht]
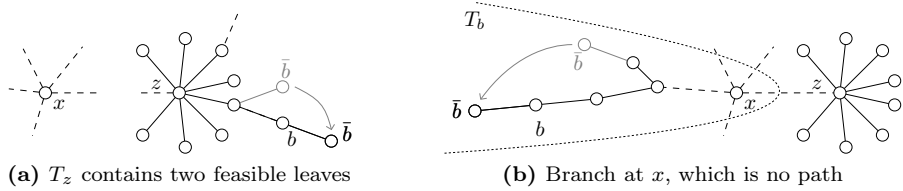

 \centering
 \subfloat[$T_z$ contains two feasible leaves]{
  \begin{pgfpicture}{3.2cm}{0.8cm}{8cm}{2.8cm}
   \begin{pgfmagnify}{0.8}{0.8}

\pgfnodecircle{Node1}[stroke]{\pgfxy(7,2)}{0.1cm}
\pgfputat{\pgfnodeborder{Node1}{1600}{0.3cm}}{\pgfbox[center,center]{$z$}}
\pgfnodecircle{Node2}[stroke]{\pgfxy(7.1,2.9)}{0.1cm}
\pgfnodecircle{Node3}[stroke]{\pgfxy( 7.7 ,2.7 )}{0.1cm}
\pgfnodecircle{Node4}[stroke]{\pgfxy(7.9,2.2)}{0.1cm}
\pgfnodecircle{Node5}[stroke]{\pgfxy(7.9,1.8)}{0.1cm}
\pgfnodecircle{Node6}[stroke]{\pgfxy(7.7,1.3)}{0.1cm}
\pgfnodecircle{Node7}[stroke]{\pgfxy(7.1,1.1)}{0.1cm}
\pgfnodecircle{Node8}[stroke]{\pgfxy(6.4,1.3)}{0.1cm}
\pgfnodecircle{Node9}[stroke]{\pgfxy(6.4,2.7)}{0.1cm}

\pgfnodecircle{Node22}[stroke]{\pgfxy(8.7,1.5)}{0.1cm}
\pgfputat{\pgfnodeborder{Node22}{300}{0.2cm}}{\pgfbox[center,center]{$b$}}

\textcolor{black}{
\pgfnodecircle{Node16}[stroke]{\pgfxy(9.5,1.2)}{0.1cm}
\pgfnodeconnline{Node22}{Node16}
\pgfputat{\pgfnodeborder{Node16}{30}{0.2cm}}{\pgfbox[center,center]{$\bar b$}}
\pgfnodecircle{Node16}[stroke]{\pgfxy(9.5,1.2)}{0.1cm}
\pgfnodeconnline{Node22}{Node16}
\pgfputat{\pgfnodeborder{Node16}{30}{0.2cm}}{\pgfbox[center,center]{$\bar b$}}
}
\textcolor{gray}{
\pgfnodecircle{Node23}[stroke]{\pgfxy(8.7,2.1)}{0.1cm}
\pgfnodeconnline{Node5}{Node23}
\pgfputat{\pgfnodeborder{Node23}{80}{0.2cm}}{\pgfbox[center,center]{$\bar b$}}
  \pgfsetendarrow{\pgfarrowpointed}
  \pgfxycurve(8.9,2.1)(9.3,1.9)(9.4,1.8)(9.5,1.4)
}
\pgfnodecircle{Node15}[stroke]{\pgfxy(4.8,2)}{0.1cm}
\pgfputat{\pgfnodeborder{Node15}{320}{0.2cm}}{\pgfbox[center,center]{$x$}}

\pgfnodecircle{Node17}[virtual]{\pgfxy(3.9,2.1)}{0.3cm}
\pgfnodecircle{Node18}[virtual]{\pgfxy(6,2)}{0.3cm}
\pgfnodecircle{Node19}[virtual]{\pgfxy(5.6,3)}{0.3cm}
\pgfnodecircle{Node20}[virtual]{\pgfxy(4.3,3)}{0.3cm}
\pgfnodecircle{Node21}[virtual]{\pgfxy(4.5,1)}{0.3cm}
\pgfnodecircle{Node24}[virtual]{\pgfxy(8,3.4)}{0.1cm}

\pgfnodeconnline{Node1}{Node2}
\pgfnodeconnline{Node1}{Node3}
\pgfnodeconnline{Node1}{Node4}
\pgfnodeconnline{Node1}{Node5}
\pgfnodeconnline{Node1}{Node6}
\pgfnodeconnline{Node1}{Node7}
\pgfnodeconnline{Node1}{Node8}
\pgfnodeconnline{Node1}{Node9}
\pgfnodeconnline{Node22}{Node5}

\pgfsetdash{{3pt}{3pt}}{0pt}
\pgfnodeconnline{Node15}{Node17}
\pgfnodeconnline{Node15}{Node18}
\pgfnodeconnline{Node1}{Node18}
\pgfnodeconnline{Node15}{Node19}
\pgfnodeconnline{Node15}{Node20}
\pgfnodeconnline{Node15}{Node21}
\pgfnodeconnline{Node22}{Node5}
\pgfnodeconnline{Node24}{Node3}

  \end{pgfmagnify}
  \end{pgfpicture}
\label{fig:SternTransformation_0}
}\hfill
\subfloat[Branch at $x$, which is no path]{
  \begin{pgfpicture}{0cm}{0.8cm}{6cm}{2.9cm}
   \begin{pgfmagnify}{0.8}{0.8}

\pgfnodecircle{Node1}[stroke]{\pgfxy(6.5,2)}{0.1cm}
\pgfputat{\pgfnodeborder{Node1}{1600}{0.3cm}}{\pgfbox[center,center]{$z$}}
\pgfnodecircle{Node2}[stroke]{\pgfxy(6.6,2.9)}{0.1cm}
\pgfnodecircle{Node3}[stroke]{\pgfxy( 7.2 ,2.7 )}{0.1cm}
\pgfnodecircle{Node4}[stroke]{\pgfxy(7.4,2.2)}{0.1cm}
\pgfnodecircle{Node5}[stroke]{\pgfxy(7.4,1.8)}{0.1cm}
\pgfnodecircle{Node6}[stroke]{\pgfxy(7.2,1.3)}{0.1cm}
\pgfnodecircle{Node7}[stroke]{\pgfxy(6.6,1.1)}{0.1cm}
\pgfnodecircle{Node8}[stroke]{\pgfxy(5.9,1.3)}{0.1cm}
\pgfnodecircle{Node9}[stroke]{\pgfxy(5.9,2.7)}{0.1cm}

\pgfnodecircle{Node10}[stroke]{\pgfxy(3.5,2.1)}{0.1cm}
\pgfnodecircle{Node11}[stroke]{\pgfxy(2.5,1.9)}{0.1cm}
\pgfnodecircle{Node13}[stroke]{\pgfxy(3.1,2.5)}{0.1cm}

\pgfnodecircle{Node12}[stroke]{\pgfxy(1.5,1.8)}{0.1cm}
\pgfputat{\pgfnodeborder{Node12}{280}{0.3cm}}{\pgfbox[center,center]{$b$}}

\textcolor{black}{
\pgfnodecircle{Node16}[stroke]{\pgfxy(0.5,1.7)}{0.1cm}
\pgfnodeconnline{Node12}{Node16}
\pgfputat{\pgfnodeborder{Node16}{175}{0.2cm}}{\pgfbox[center,center]{$\bar b$}}
\pgfnodecircle{Node16}[stroke]{\pgfxy(0.5,1.7)}{0.1cm}
\pgfnodeconnline{Node12}{Node16}
\pgfputat{\pgfnodeborder{Node16}{175}{0.2cm}}{\pgfbox[center,center]{$\bar b$}}
}

\textcolor{gray}{
\pgfnodecircle{Node14}[stroke]{\pgfxy(2.3,2.8)}{0.1cm}
\pgfnodeconnline{Node13}{Node14}
\pgfputat{\pgfnodeborder{Node14}{250}{0.2cm}}{\pgfbox[center,center]{$\bar b$}}
  \pgfsetendarrow{\pgfarrowpointed}
  \pgfxycurve(2.1,2.8)(1.7,2.8)(1.1,2.6)(0.6,1.9)
}
\pgfnodecircle{Node15}[stroke]{\pgfxy(4.8,2)}{0.1cm}
\pgfputat{\pgfnodeborder{Node15}{320}{0.2cm}}{\pgfbox[center,center]{$x$}}

\pgfnodecircle{Node17}[virtual]{\pgfxy(3.9,2.1)}{0.3cm}
\pgfnodecircle{Node18}[virtual]{\pgfxy(6,2)}{0.3cm}
\pgfnodecircle{Node19}[virtual]{\pgfxy(5.6,3)}{0.3cm}
\pgfnodecircle{Node20}[virtual]{\pgfxy(4.3,3)}{0.3cm}
\pgfnodecircle{Node21}[virtual]{\pgfxy(4.5,1)}{0.3cm}

\pgfnodeconnline{Node1}{Node2}
\pgfnodeconnline{Node1}{Node3}
\pgfnodeconnline{Node1}{Node4}
\pgfnodeconnline{Node1}{Node5}
\pgfnodeconnline{Node1}{Node6}
\pgfnodeconnline{Node1}{Node7}
\pgfnodeconnline{Node1}{Node8}
\pgfnodeconnline{Node1}{Node9}
\pgfnodeconnline{Node10}{Node11}
\pgfnodeconnline{Node12}{Node11}
\pgfnodeconnline{Node10}{Node13}

\pgfsetdash{{3pt}{3pt}}{0pt}
\pgfnodeconnline{Node15}{Node10}
\pgfnodeconnline{Node15}{Node1}
\pgfnodeconnline{Node15}{Node19}
\pgfnodeconnline{Node15}{Node20}
\pgfnodeconnline{Node15}{Node21}

\pgfputat{\pgfxy(0.5,3.2)}{\pgfbox[center,center]{$T_b$}}
\pgfsetdash{{1pt}{1pt}}{0pt}
\pgfxycurve(0,1)(7,1.5)(7.5,2.3)(0.5,3.5)

  \end{pgfmagnify}
  \end{pgfpicture}
\label{fig:SternTransformation_1}
}
\caption{$b$ and $\overline{b}$ can be chosen in same branch at $x$}

\end{figure}

%% file: boundsInGeneralGraphs.tex
\section{Bounds for general graphs}
\label{sec:generalgraphs}

Let $G$ be an arbitrary undirected, simple, connected  graph. Let
$T$ be an arbitrary spanning tree of $G$. Then $\dist_G(x,y) \leq
\dist_T(x,y)$ holds for all vertices $x,y$ and therefore
 $\rad(G) \leq \rad(T)$ and $\s(G) \leq \s(T)$. By choosing the spanning 
 tree of the same maximum degree as $G$ we get the upper bound.
 
 On the other hand,
$G$ contains spanning trees $T_1$, $T_2$ such that $\rad(G) =
\rad(T_1)$ and $\s(G) = \s(T_2)$ hold. Thus the radius and status
of $G$ is bounded by the radius and status of certain spanning
trees of $G$. 
However, these spanning trees do not need to have the same maximum degree as $G$. 
But the following lemma holds. Note that the assertion on the status can be directly derived 
from a lemma from Lin et al.~in \cite{LinShangZhang11}.

\begin{Lem}
 Let $2\leq \ell \leq k\leq n$. Then 
 \begin{align*}
  \s(B_{n,k}) \leq \s(B_{n,\ell}) 
 \end{align*}
 with equality if and only if $k= \ell$ and
 \begin{align*}
  \rad(B_{n,k}) \leq \rad(B_{n,\ell}).
 \end{align*}
\end{Lem}

\begin{proof}
 To prove these statements we can use again the transformation introduced in Section \ref{sec:transformation}. Let $T$ be a tree with $\Delta(T) = \ell$. We will show that if $\ell < k$ we can transform $T$ into a tree $\overline{T}$ whose maximum degree is still less than or equal to $k$ and its status (resp. radius) is less than (or equal to) the corresponding value of $T$. Start with $T= B_{n,\ell}$ and iterate until the maximum degree of the resulting tree is equal to $k$. The assertion then follows from Theorem \ref{lowerboundstatus} (resp. \ref{lowerboundexc}), that is
 \begin{align*}
  \s(B_{n,\ell})>&\s(\overline T)\geq \s(B_{n,k}) \text{ and }\\
  \rad(B_{n,\ell})\geq&\rad(\overline T)\geq \rad(B_{n,k})
 \end{align*}

 First we consider the status, let $x$ be a centroid vertex and $b$ a leaf of $T$ with $\dist(x,b) = \ecc(x)$ (since $k>\ell$ we know that $\ecc(x) >1$). Let $\overline T$ be the  tree resulting from the removal of the edge connecting $b$ and insertion of $(x,b)$. According to Proposition \ref{transformationCentroid2}, $x$ is centroid vertex of $T$. Further $\Delta(\overline T)\leq k$.
  Similar to the calculations in \eqref{xBleibtZentroid} we get $\s(\overline T) < \s(T)$.

 Now let us assume that $x$ is a central vertex and transform $T$ in the same manner as above. Clearly the radius does not increase during the transformation. This completes the proof.

\end{proof}

These observations  lead to the following two theorems. Note that
the second theorem has already been stated in
\cite{LinShangZhang11} and is presented here to demonstrate the connection between 
radius and status explicitly.

\begin{The}
Let $G = (V,E)$ be an undirected simple, connected graph  with $n$
vertices and maximum degree $\Delta(G)=k$. Then
\begin{align*}
 \rad(B_{n,k}) \leq \rad(G) \leq \rad(S_{n,k}).
\end{align*}
In particular we get
\begin{enumerate}
 \item[(i)] If $G$ contains a $k$-balanced tree $B_{n,k}$, then $\rad(G) = \rad(B_{n,k})$.
 \item[(ii)] If $\rad(G) = \rad(S_{n,k})$, then $G$ contains a
tree $C^{\star}_{n,k}$.
\end{enumerate}

\end{The}

\begin{The}(Lin et al.~2011, \cite{LinShangZhang11})
Let $G = (V,E)$ be an undirected simple, connected graph  with $n$
vertices and maximum degree $\Delta(G)=k$. Then
\begin{align*}
 \s(B_{n,k}) \leq \s(G) \leq \s(S_{n,k}).
\end{align*}
In particular we get
\begin{enumerate}
 \item[(i)] $\s(G) = \s(B_{n,k})$ if and only if $G$ contains a $k$-balanced tree $B_{n,k}$.
 \item[(ii)] If $\s(G) = \s(S_{n,k})$ then $G$ contains a comet $S_{n,k}$.
\end{enumerate}

\end{The}

%% file: conclusions.tex
\section{Conclusions and outlook}
\label{sec:concl}

The results of this work present sharp lower and upper bounds on the radius and status of an undirected, connected, unweighted graph $G$. As a connected graph, $G$ contains spanning trees of the same order but not necessarily of the same maximum degree. The distance of two vertices in $G$ is less than or equal to the distance of the two in each spanning tree of $G$. Therefore both radius and status of $G$ are less than or equal to the radius and status, resp., of each spanning tree of $G$. On the other hand there exists a spanning tree $T$ with the same distances as in 	$G$ and therefore the radius and status of $G$ equals to the radius and status, resp., of $T$. But this tree $T$ might have a smaller maximum degree than $G$. However, the higher the maximum degree, the lower the values of radius and status.
To give lower and upper bounds for radius and status of $G$ with respect to its order and maximum degree, it suffices to investigate the two functions on trees with the same order and maximum degree. 

Although Lin et al.~published the bounds on the status and Vizing pointed out the upper bound of the radius, this work specifically demonstrates the similarities of the extremal behavior of status and radius on graphs. Both functions take their minimum and maximum on the same type of tree graphs. Further the results were obtained by a new proof technique which regains the known results and proves the new ones in a simple and elegant manner.

So far, the edges of the graphs considered in this work had all length one. As a next step it is natural to investigate bounds of the radius and the status if it is allowed to assign positive length to the edges, that is, a function $\ell:E\longrightarrow \mathbb{R}_{>0}$. The distance $\dist(x,y)$ of two vertices $x$ and $y$ is length of a shortest path from $x$ to $y$, as usual. The definitions of status and radius remain the same, only the notion of distance becomes more general. The question is whether we can give sharp bounds and if the extremal behavior of status and radius remains to be similar. Depending on the context there might be restrictions on the distance of certain vertices to each other which need to be considered additionally. It is not clear if the transformation described in Section \ref{sec:transformation} can be adapted to this case. A new approach may be required.

%% file: acknowledgement.tex
\section*{Acknowledgement}

We thank three anonymous colleagues for their constructive remarks on an earlier version of this paper.